\documentclass{amsart}
\usepackage[T1]{fontenc}
\usepackage{ae}
\usepackage{aecompl}
\usepackage{ucs}
\usepackage[utf8x]{inputenc}
\usepackage{color}
\usepackage[enableskew,vcentermath]{youngtab}
\usepackage{amsfonts}
\usepackage{hhline}
\usepackage{amsthm}
\usepackage{amsmath}
\usepackage{amssymb}

\begin{document}

\title[Principal Hook Lengths And Durfee Sizes]{On Principal Hook Length Partitions And Durfee Sizes In Skew Characters}
\author[C. Gutschwager]{Christian Gutschwager}
\address{Institut für Algebra, Zahlentheorie und Diskrete Mathematik, Leibniz Universität Hannover,  Welfengarten 1, D-30167 Hannover}
\email{gutschwager (at) math (dot) uni-hannover (dot) de}

\newtheorem{Le}{Lemma}[section]
\newtheorem{Ko}[Le]{Lemma}
\newtheorem{Sa}[Le]{Theorem}
\newtheorem{pro}[Le]{Proposition}
\newtheorem{Bem}[Le]{Remark}
\newtheorem{Def}[Le]{Definition}
\newtheorem{Bsp}[Le]{Example}
\renewcommand{\l}{\lambda}
\newcommand{\bl}{\bar\lambda}
\newcommand{\bn}{\bar\nu}
\newcommand{\mA}{\mathcal{A}}
\newcommand{\mB}{\mathcal{B}}
\newcommand{\mC}{\mathcal{C}}\newcommand{\mD}{\mathcal{D}}
\renewcommand{\a}{\alpha}
\renewcommand{\b}{\beta}
\renewcommand{\pm}[1]{\begin{pmatrix}#1\end{pmatrix}}
\newcommand{\pinw}{{\pi_{nw}}}
\newcommand{\abs}[1]{\lvert #1 \rvert}
\newcommand{\tm}{\tilde\mu}
\newcommand{\tn}{\tilde\nu}
\newcommand{\m}{\mu}
\newcommand{\n}{\nu}
\newcommand{\lm}{\l/\m}
\newcommand{\h}{\hfil}


\begin{abstract}
We construct for a given arbitrary skew diagram $\mA$ all partitions $\nu$ with maximal principal hook lengths among all partitions with $[\nu]$ appearing in $[\mA]$. Furthermore we show that these are also partitions with minimal Durfee size.

We use this to give the maximal Durfee size for $[\nu]$ appearing in $[\mA]$ for the cases when $\mA$ decays into two partitions and for some special cases of $\mA$. We also deduce necessary conditions for two skew diagrams to represent the same skew character.

\end{abstract}
\maketitle

\section{Introduction}
Examining the decomposition of a skew character $[\mA]$ into irreducible characters $[\mA]=\sum_\n c_\n[\n]$ there is much interest in knowing as much as possible about the $[\nu]$ which appear in $[\mA]$. It is known for example that rearranging the parts respectively heights of $\mA$ into a partition gives the lexicographically smallest resp.\ largest  partition $\pi_{min}$ resp.\ $\pi_{max}$ and both appear with multiplicity $1$. We will show in the following how to obtain from the northwest ribbon decomposition of $\mA$ those $[\nu]$ appearing in $[\mA]$ with $\nu$ having lexicographically largest principal hook lengths. Furthermore we give an easy formula for the multiplicities of those $[\nu]$.

We construct for a skew character $[\lambda/\mu]=\sum_\nu c(\lambda;\mu,\nu) [\nu]$ the $\nu$ with maximal principal hook lengths, which are the hook lengths of the boxes $(i,i)$ of $\nu$. From this we deduce the minimal Durfee size for characters in arbitrary skew characters and the maximal Durfee size for characters in products of characters and some special skew characters.

We start in Section~\ref{Se:disc} with skew diagrams $\mA$ which decompose into two disconnected proper diagrams $\a,\b$ and so are equivalent to the product \[([\a]\times[\b])\uparrow_{S_n\times S_m}^{S_{n+m}}=:[\a]\otimes[\b].\] We construct the $\nu$ with maximal principal hook lengths and show that the corresponding $[\nu]$ all appear with multiplicity $1$ in $[\mA]$. This gives also the minimal Durfee size for $[\m]\in[\mA]$.

In Section~\ref{Se:main} we generalize Section~\ref{Se:disc} to arbitrary skew diagrams.

In Section~\ref{Se:maxD} we use the results from Sections~\ref{Se:disc} and~\ref{Se:main} and \cite[Theorem 4.2]{Gut} to construct for the product $[\a]\otimes[\b]$ and some special skew characters some $\nu$ with maximal Durfee size. In particular we can easily calculate the maximal Durfee size in those cases.

There has recently been much interest in the question of determining necessary or sufficient conditions for two skew diagrams $\mA,\mB$ to have either $[\mA]-[\mB]$ positive or even $[\mA]=[\mB]$, see for example~\cite{MN},\cite{MW},\cite{RSW}.
In Section~\ref{Se:eq} we use the results from Section~\ref{Se:main} to give necessary conditions for two skew diagrams $\mA$ and $\mB$ to represent the same skew character, i.e.\ $[\mA]=[\mB]$.

Skew characters $[\mA]$ are strongly related to skew Schur functions $s_\mA$ (see\cite{Sag}).

\section{Notation and Littlewood-Richardson-Symmetries}
We mostly follow the standard notation in \cite{Sag}. A partition $\l=(\l_1,\l_2,\ldots,\l_l)$ is a weakly decreasing sequence of non-negative integers $\l_i$, the parts of $\l$. With a partition $\l$ we associate a diagram, which we also denote by $\l$, containing $\l_i$ left-justified boxes in the $i$-th row and we use matrix-style coordinates to refer to the boxes. The length $l(\l)=l$ of a partition is the number of positive parts $\l_i$ of $\l$ and for the number of boxes in $\l$ we write $\left|\l\right|=\sum_i \l_i$.

The conjugate $\l'$ of $\l$ is the diagram which has $\l_i$ boxes in the $i$-th column. The sum $\mu+\nu=\l$ of two partitions $\mu,\nu$ is defined by $\l_i=\mu_i+\nu_i$. As always we assume $\l_i=0$ for $i>l(\l)$.

For $\mu \subseteq \l$ we define the skew diagram $\l/\mu$ as the difference of the diagrams $\l$ and $\mu$ defined as the difference of the sets of boxes. Rotation of $\l/\mu$ by $180^\circ$ yields a skew diagram $(\l/\mu)^\circ$ which is well defined up to translation. A skew tableau $T$ is a skew diagram in which the boxes are replaced by positive integers.  We refer to the entry in box $(i,j)$ as $T(i,j)$. A semistandard Young tableau of shape $\l/\mu$ is a filling of $\l/\mu$ with positive integers such that the following expressions hold for all $(i,j)$ for which they are defined: $T(i,j)<T(i+1,j)$ and $T(i,j)\leq T(i,j+1)$. The content of a semistandard tableau $T$ is $\nu=(\nu_1,\ldots)$ if the number of occurrences of the entry $i$ in $T$ is $\nu_i$. The reverse row word of a tableau $T$ is the sequence obtained by reading the entries of $T$ from right to left and top to bottom starting at the first row. Such a sequence is said to be a lattice word if for all $i,n \geq1$ the number of occurrences of $i$ among the first $n$ terms is at least the number of occurrences of $i+1$ among these terms. The Littlewood-Richardson (LR) coefficient $c(\l;\mu,\nu)$ equals the number of semistandard tableaux of shape $\l/\mu$ with content $\nu$ such that the reverse row word is a lattice word. We will call those tableaux LR-tableaux. The LR-coefficients play an important role in different contexts (see \cite{Sag}).

 The irreducible characters $[\l]$ of the symmetric group $S_n$ are naturally labeled by partitions $\l\vdash n$. The skew character $[\lm]$ corresponding to a skew diagram $\lm$ is defined by the LR-coefficients
\[ [\lm]=\sum_\n c(\l;\m,\n) [\n] \] and we write $[\nu]\in[\l/\mu]$ iff $c(\l;\mu,\nu)\neq 0$.

Some well known relations are the following:

We have that $c(\l;\mu,\nu)=c(\l;\nu,\mu)$ and $[\mA]=[\mA^\circ]$. If the skew diagrams $\l/\mu$ and $\alpha/\beta$ are the same up to translation we have $c(\l;\mu,\nu)=c(\alpha;\beta,\nu)$ for every $\nu$.

We say that a skew diagram $\mathcal{D}$ decomposes into the disconnected skew diagrams $\mathcal{A}$ and $\mathcal{B}$ if no box of $\mathcal{A}$ (viewed as boxes in $\mathcal{D}$) is in the same row or column as a box of $\mathcal{B}$. If $\mathcal{D}$ does not decompose we call it connected.

A skew character whose skew diagram $\mathcal{D}$ decomposes into disconnected (skew) diagrams $\mathcal{A},\mathcal{B}$ is equivalent to the product of the characters of the disconnected diagrams induced to a larger symmetric group. We have  \[[\mD]=([\mA]\times[\mB])\uparrow_{S_n\times S_m}^{S_{n+m}}=:[\mA]\otimes[\mB]\] with $\abs{\mA}=n,\abs{\mB}=m$.  If $\mD=\l/\mu$ and $\mA,\mB$ are proper partitions $\a,\b$ we have:
\[[\l/\mu]= \sum_\nu c(\l;\mu,\nu)[\nu]=\sum_\nu c(\nu;\a,\b)[\nu] =[\a]\otimes[\b].\]

In the cohomology ring $H^*(Gr(l,\mathbb{C}^n),\mathbb{Z})$ of the Grassmannian $Gr(l,\mathbb{C}^n)$ of $l$-di\-men\-sio\-nal subspaces of $\mathbb{C}^n$ the product of two Schubert classes $\sigma_\a, \sigma_\b$ is given by:

\[\sigma_\a\cdot\sigma_\b=\sum_{\nu\subseteq((n-l)^l)}c(\nu;\a,\b)\sigma_\nu.\]

In~\cite[Section4]{Gut} we established a close connection between the Schubert-Product and skew characters. To use this relation later on we define the Schubert-Product for characters in the obvious way as a restriction of the ordinary product:

\[[\a]\star_{(k^l)}[\b]:=\sum_{\nu\subseteq(k^l)} c(\nu;\a,\b)[\nu]. \]

On the set of partitions we define an order, the lexicographic order, and say that $\mu$ is smaller than $\nu$, $\mu<\nu$, if there is an $i$ with $\mu_i<\nu_i$ and for all $1\leq j \leq i-1$ we have $\mu_j=\nu_j$.

A hook is a partition which does not contain the subdiagram $(2^2)$ and so is of the form $(r,1^s)$. For each box $(i,j)$ in a diagram $\l$ we define its arm respectively leg length as the number of boxes to the right resp.\  below of it in the same row resp.\  column. The hook length of a box is the sum of the arm and leg lengths plus $1$ (for the box itself).

A (proper) ribbon is a connected skew diagram which does not contain the subdiagram $(2^2)$. A (disconnected) skew diagram which decomposes into ribbons will be called a weak ribbon.

For a ribbon $\mathcal{R}$, we define its arm resp.\  leg length as the number of columns resp.\  rows in $\mathcal{R}$ minus $1$. The arm resp.\  leg length of a weak ribbon $\mathcal{R}$  is defined as the sum of the arm resp.\  leg lengths of the ribbons into which  $\mathcal{R}$ decomposes, which is the number of columns resp.\  rows in  $\mathcal{R}$ minus the number of ribbons into which  $\mathcal{R}$ decomposes.

For each (connected) skew diagram  $\mA$ we define its first northwest ribbon $nw_1(\mA)$ as the subdiagram which starts in the lowest leftmost box traverses along the northwest border of $\mA$ and ends in the box in the top right. To get the second northwest ribbon $nw_2(\mA)$ we remove $nw_1$ from $\mA$ and repeat this process if $\mA/nw_1$ is still connected. This we iterate to get $nw_i(\mA)$. For a disconnected skew diagram which decays into two or more skew diagrams $\mB_j$ we define its northwest ribbons as the weak ribbons which contain the corresponding northwest ribbons of the $\mB_j$. All the northwest ribbons together form the northwest ribbon decomposition. Furthermore we define the northwest ribbon length partition $\pinw(\mA)$ associated to $\mA$ as the partition where the $i$th row has as many boxes as the $i$th northwest ribbon $nw_i(\mA)$: \[\pinw(\mA)_i=\left| nw_i(\mA)\right|.\] The northwest ribbons are weak ribbons and only in some cases proper ribbons.

\begin{Bsp}\label{ex1}
We give the northwest ribbon decomposition for the skew diagrams $(10^2,8^4,5^2)/(5^4)$ and $(10^4,8^2,3^2)/(5^4)$ and label the boxes with $i$ if they are contained in the $i$-th northwest ribbon $nw_i$:
{\footnotesize\[\young(:::::11111,:::::12222,:::::123,:::::123,11111123,12222223,12333,12344)\qquad \young(:::::11111,:::::12222,:::::12333,:::::12344,11111123,12222223,123,123).\]}
In both cases the third northwest ribbon $nw_3$ decays into two ribbons and $\pinw=(17,15,8,2)$.
\end{Bsp}

\begin{Def}
  For a partition $\l$ we define the $i$th principal hook length $hl_i(\l)$ as the hook length of the box $(i,i)$ and the (principal) hook length partition as: $hl(\l)=(hl_1(\l),hl_2(\l),\ldots)$. So we have $hl_1(\l)=\l_1+l(\l)-1$.

  For a skew diagram $\mA$ we define its hook length partition $hl(\mA)$ as the lexicographic biggest hook length partition $hl(\l)$ of all $\l$ with $[\l]$ appearing in $[\mA]$:
\[hl(\mA)=hl([\mA])=\max{}_{\textnormal{lex}}(hl(\nu)\,|\,[\nu]\in[\mA]).\]
\end{Def}

In a partition the $i$th northwest ribbon is the hook to the box $(i,i)$ so for a partition $\l$ we have $hl(\l)=\pinw(\l)$.

The Durfee size $d(\l)$ of a partition $\l$ is $d$ if $(d^d)\subseteq\l$ is the largest square contained in $\l$. This means that $hl_{d(\l)}(\l)\not=0$ but $hl_{d(\l)+1}(\l)=0$. From this follows $d(\l)=l(hl(\l))$. The Durfee size of a skew diagram $\mA$ or skew character $[\mA]$ is the biggest Durfee size of all partitions whose corresponding character appears in the decomposition of $[\mA]$:
\[d(\mA)=d([\mA])=\max(d(\nu)\,|\,[\nu]\in[\mA]).\]

\section{Maximal Hook Lengths In Products Of Irreducible Characters}\label{Se:disc}
In this section we study the case when $\mA$ decomposes into two proper partitions $\a,\b$ and show $hl(\mA)=\pinw(\mA)$. We will show that there are $2^{\min(d(\a),d(\b))}$ partitions $\nu$ with $[\nu]\in [\mA]$ and $hl(\n)=hl(\mA)$ and show that each of those $[\nu]$ appears with multiplicity $1$ in $[\mA]$. In Remark~\ref{bem:gamma} we will determine the exact shape of all these $\nu$. Furthermore we argue in Remark~\ref{bem:minDdisc} that those $[\nu]$ with $hl(\nu)=hl(\mA)$ have minimal Durfee size of all $[\tn]\in[\mA]$.

\begin{Sa} \label{Sa:discmain}
 Let $\mA$ be a skew partition decomposing into the proper partitions $\a$ and $\b$.

Then $hl(\mA)=hl(\a)+hl(\b)=\pinw(\mA)$.

Furthermore there are $2^{\min(d(\a),d(\b))}$ partitions $\nu$ with $hl(\nu)=hl(\mA)$ and $[\nu]$ appearing in $[\mA]$. All these $[\nu]$ appear with multiplicity $1$.
\end{Sa}
\begin{proof}
 If $\mA=\l/\mu$ decomposes into two partitions $\a$ and $\b$  we have $[\mA]=[\a]\otimes [\b]$. If we decompose $[\mA]=\sum_\nu c(\l;\mu,\nu)[\nu]=\sum_\nu c(\nu;\a,\b)[\nu]$ we have to create LR-tableaux with shape $\nu/\alpha$ and content $\b$.

To create an LR-tableau $T$ with shape $\nu/\a$ and content $\b$ and $\nu$ having the lexicographic biggest hook length partition of all $\bar\nu$ with $[\bar\nu]$ appearing in $[\mA]$ we have to fill in $T$ as many boxes as possible in the first row and column. Because of the LR conditions we can only place only the entry $1$ into the boxes of the first row and so we maximize $\nu_1$ by placing all the $\b_1$ entries $1$ into the first row and get $\nu_1=\a_1+\b_1$. Into the first column we can only place each entry once (but the entries $1$ are used up already), so we get a maximized first column by filling the boxes with entries $2$ to $l(\b)$.  So we have $l(\a)+l(\b)-1=l(\nu)$ boxes in the first column and this gives us:
\[hl_1(\mA)=hl_1(\nu)=\a_1+\b_1+l(\a)+l(\b)-1-1=hl_1(\a)+hl_1(\b).\]
Another way to obtain an LR-tableau $U$ with shape $\tilde\nu/\a$ and content $\b$ and $hl_1(\tilde\nu)=hl_1(\nu)$ would be to place the entries $1$ to $l(\b)$ into the first column and the remaining $\b_1-1$ entries $1$ into the first row. Clearly these are the only ways to maximize $hl_1(\nu)$.

We show that this can be iterated by examining the filling of $T$ (and $U$) which maximizes $hl_2(\nu)$. Without loss of generality we may assume $d(\b)\leq d(\a)$.

If $d(\b)=1$ then we are finished and have no entries anymore to place in $T$ (or $U$).

If $d(\b)\geq1$ then we have again two possibilities to maximize $hl_2(\nu)$ by either maximizing the second row or column. The fillings with maximized second row differs from the filling with maximized second column because the box $(2,2)$ belongs to $\a$ and so remains empty in $T$.
We have only to show that both possibilities satisfy the LR-conditions. Because there are only $\b_2-1$ entries $2$ left to place in $T$ (or $U$) and there are either $\b_1\geq \b_2-1$ ($T$) or $\b_1-1\geq\b_2-1$ ($U$) entries $1$ placed in the first row of $T$ (or $U$) the LR conditions are satisfied for the entries $2$ no matter where we place them. For the entries placed in the second column the LR lattice condition clearly is satisfied but we have to check that the entries are weakly decreasing amongst the rows. If we assume that the box $(j,2)$ is filled, its entry is $j-\a'_2+1$ ($+1$ if all remaining entries $2$ are placed in the second row) with $\a'_2$ the length of the second column of $\a$. The box $(j,1)$ is empty in the case $j\leq \a'_1$ or otherwise has the entry $j-\a'_1$ ($+1$ for $T$). By comparing the worst cases we have the condition: $j-\a'_1+1\leq j-\a'_2+1$ which holds for all $\a$.
Furthermore the maximized second row has no more than $\a_2+\b_2-1$ boxes which is not bigger than the number of boxes in the first row since there are $\a_1+\b_1$ ($-1$ for $U$) boxes in the first row, so $\nu$ is still a partition. The same reasoning applies to the columns.

So all works well and we can iterate the process. In the end we had $d(\b)=l(hl(\b))$ choices to make, to either maximize the $i$th row or $i$th column for $i\leq d(\b)$, and so get $2^{\min(d(\a),d(\b))}$ different partitions $\nu$. For each such $\nu$, there is a unique (so the multiplicity of $[\nu]$ in $[\mA]$ is $1$) LR-tableau $T$ of shape $\nu/\a$ and content $\b$.
\end{proof}

\begin{Bem}\label{bem:gamma}
The proof tells us even more about the explicit form of the $[\nu]$ appearing in $[\mA]$ with $hl(\nu)=hl(\mA)$.

Let $\gamma$ be the partition such that the $i$th principal hook has as arm resp.\  leg length the sum of the arm resp.\  leg lengths of the $i$th principal hooks of $\a$ and $\b$ and containing the box $(i,i)$ exactly if $i\leq\max(d(\a),d(\b))$. Then $\gamma$ is the intersection of all $\nu$ with maximal hook length partition and also the intersection of the $\nu$ where either all columns or all rows were maximized as described above.

From $\gamma$ we can construct all partitions $\nu$ appearing in $\mA$ with maximal hook length partition by adding for each $1\leq j\leq \min(d(\a),d(\b))$ a box to either the $j$th row or column of $\gamma$.\end{Bem}

\begin{Bsp}
If $\mA$ decomposes into the partitions $\a=(5^2,4^2,3,1)$ and $\b=(5,3^2,2,1^2)$
{\footnotesize \[\a=\yng(5,5,4,4,3,1) \qquad \b=\yng(5,3,3,2,1,1)\] }
the characters corresponding to the following partitions $\nu$ are the ones with maximal hook length partition in $[\mA]$. In $\bar\gamma$ the unmarked boxes form the partition $\gamma$ as in Remark~\ref{bem:gamma} and the $\nu$ are obtained by choosing for each $i\in\{1,2,3\}$ exactly one box labeled $i$ and add them to $\gamma$. $T$ resp.\  $U$ gives the actual LR-filling with all rows resp.\  columns maximized:
{\footnotesize\[\bar\gamma=\young(\h\h\h\h\h\h\h\h\h1,\h\h\h\h\h\h2,\h\h\h\h3,\h\h\h\h,\h\h\h,\h\h3,\h\h,\h2,\h,\h,\h,1) \qquad
T=\young(\h\h\h\h\h11111,\h\h\h\h\h22,\h\h\h\h3,\h\h\h\h,\h\h\h,\h3,24,3,4,5,6)\qquad
U=\young(\h\h\h\h\h1111,\h\h\h\h\h2,\h\h\h\h,\h\h\h\h,\h\h\h,\h23,13,24,3,4,5,6).\]}

 For the hook length partitions we have:
\[hl(\nu)=(20,11,5,1)=(10+10,4+7,1+4,0+1)=hl(\a)+hl(\b).\]
\end{Bsp}

\begin{Bem}\label{bem:minDdisc}
The maximum of the Durfee sizes of the partitions $\a$ and $\b$ is a lower bound for the minimal Durfee size of characters in $[\a]\otimes[\b]$ and the $[\nu]\in[\a]\otimes[\b]$ with $hl(\n)=hl([\a]\otimes[\b])$ have Durfee size $d(\nu)=\max(d(\a),d(\b))$. So the characters $[\nu]$ with $hl(\n)=hl([\a]\otimes[\b])$ are characters with minimal Durfee size which is therefore given by $\max(d(\a),d(\b))$.
\end{Bem}

\section{Maximal Hook Lengths In Skew Characters}\label{Se:main}

In this section we generalize Section~\ref{Se:disc} to the case when $\mA$ is an arbitrary skew diagram and show that also in this case $hl(\mA)=\pinw(\mA)$ is true. We will give an easy formula for the coefficient of $[\nu]$ in $[\mA]$ when $\nu$ is a partition with $hl(\nu)=hl(\mA)$. In an easy way similar to that in Section~\ref{Se:disc} we can construct all those $\nu$ explicitly. Furthermore we show in Proposition~\ref{Le-prop:minD} that also in this case the characters with maximal hook length partition have minimal Durfee size.

\begin{Le}\label{sa:mainLR}
Let $\l,\nu,\bl,\bn$ be partitions such that $\nu=(\l_1,\bn+(1^{l-1}))=(\l_1,\bn_1+1,\bn_2+1,\ldots)$, $\l=(\l_1,\bl+(1^{l-1}))$ with $l=l(\l)$. (If $l(\bn)<l-1$ we again assume $\bn_i=0$ for $l(\bn)<i<l$.)

Then for all $\mu$:
\[c(\l;\mu,\nu)=c(\bl;\mu,\bn).\]
\end{Le}
\begin{proof}
We have $\bl/\bn=\l/\nu$ and so $[\bl/\bn]=[\l/\nu]$. This gives $c(\l;\mu,\nu)=c(\bl;\mu,\bn)$.
\end{proof}

\begin{Bem}\label{bem:LRmain}
So the LR fillings of $\bl/\bn$ with content $\mu$ and the LR fillings of $\l/\nu$ with content $\mu$ are the same. If $\l/\mu$ is connected, this gives us an $1-1$-correspondence between the characters $[\bn]$ in $[\bl/\mu]$ and the characters $[\nu]$ in $[\l/\mu]$ having maximal first principal hook length. In particular we have $hl_1(\mA)=\l_1+l(\l)-1$ if $\mA$ is connected.
\end{Bem}

\begin{Le}\label{bem-le:main}
 Let $\l/\mu$ be connected.

If we remove $nw_1(\l/\mu)$ from $\l/\mu$ the remaining skew diagram is $\bl/\mu$. So $\abs{nw_1(\mA)}=hl_1(\mA)$.
\end{Le}
\begin{proof}
The skew diagram $\bl/\m$ consists of those boxes $(i,j)$ that are in $\lm$ such that the box $(i-1,j-1)$ is also in $\lm$. The same is true for the skew diagram obtained by removing the first northwest ribbon of $\lm$. Furthermore we have $\abs{nw_1(\mA)}=\l_1+l(l)-1$ and so $\abs{nw_1(\mA)}=hl_1(\mA)$.
\end{proof}

\begin{Bsp}
 If we take as example $\l=(6,2,5,3^2,2^2)$ and $\mu=(3^2,2,1)$ then:{\footnotesize
\[\l/\mu: \young(\mu\mu\mu\h\h\h,\mu\mu\mu\h\h\h,\mu\mu\h\h\h,\mu\h\h,\h\h\h,\h\h,\h\h)
\qquad (\l/h)/\mu=\young(hhhhhh,h\mu\mu\mu\h\h,h\mu\mu\mu\h,h\mu\mu,h\mu\h,h\h,h\h)=\young(\mu\mu\mu111,\mu\mu\mu1\h\h,\mu\mu11\h,\mu11,11\h,1\h,1\h)=(\l/\mu)/\{nw_1(\l/\mu)\}.\]}
Here $h$ marks the boxes in $h=(\l_1,1^{l(\l)-1})$ and $1$ the boxes in $nw_1(\l/\mu)$.
\end{Bsp}

\begin{Sa}\label{Sa:main}
 Let $\mA=\l/\mu$ be a skew diagram.

 Then $hl(\mA)=\pinw(\mA)$.
\end{Sa}
\begin{proof}
We prove this by induction on the length of $\pinw(\mA)$ and the number proper ribbons into which the first northwest ribbon decays. Obviously this is true for the empty skew diagram, $l(\pinw(\mA))=0$, but we also prove this for $l(\pinw(\mA))=1$.

For $l(\pinw(\mA))=1$, $\mA$ is either a ribbon or a weak ribbon. Suppose $\mA$ is a proper ribbon. In this case we can use Lemma~\ref{bem-le:main} to get $hl_1(\mA)=\pinw(\mA)_1$. Furthermore $\pinw(\mA)_2=0$ because of the correspondence given in Remark~\ref{bem:LRmain}.

So now suppose that the claim holds if $l(\pinw(\mA))=1$ and $\mA$ decays into $j-1$ proper ribbons. Suppose now that $\mA$ decays into $j$ proper ribbons $\mB_i$ and so
\[ [\mA]=[\mB_1]\otimes[\mB_2]\otimes\cdots\otimes[\mB_j]=[\mB_1]\otimes\left([\mB_2]\otimes\cdots\otimes[\mB_j]\right). \]

By induction we know, that $hl(\mB_1)=\pinw(\mB_1)$ and $hl(\mC)=\pinw(\mC)$ with $\mC=\bigotimes_{2\leq i \leq j} \mB_i$.

We decompose $[\mB_1]=\sum_a [\n_a]$ and $[\mC]=\sum_b [\xi_b]$ into sums of irreducible characters $[\n_a]$ resp.\ $[\xi_b]$ with $\n_a$ and $\xi_b$ proper partitions. We now have $[\mA]=\sum_{a,b}[\n_a]\otimes[\xi_b]$. By Theorem~\ref{Sa:discmain} we have $hl([\n_a]\otimes[\xi_b])=hl(\n_a)+hl(\xi_b)$. So $hl(\mA)=\max_{a,b}\bigl(hl(\n_a)+hl(\xi_b)\bigr)$ but since $\n_a$ and $\xi_b$ are independent we have \[hl(\mA)=\max_{a,b}\bigl(hl(\n_a)+hl(\xi_b)\bigr)=\max_a hl(\n_a)+\max_b hl(\xi_b) = hl(\mB_1)+hl(\mC). \] But by induction $hl(\mB_1)=\pinw(\mB_1)$ and $hl(\mC)=\pinw(\mC)$ and so in total $hl(\mA)=\pinw(\mB_1)+\pinw(\mC)$. But by the definitions of the northwest ribbons and $\pinw$ we have $\pinw(\mA)=\pinw(\mB_1)+\pinw(\mC)$ if $\mA=\mB_1\otimes\mC$ as in this case. This gives finally $hl(\mA)=\pinw(\mA)$ for $l(\pinw(\mA))=1$.

So let us assume that the claim holds for $l(\pinw(\mA))=i-1$. Suppose now that $l(\pinw(\mA))=i$.

Let $\mA$ be connected. Lemma~\ref{sa:mainLR} tells us that $\left|nw_1(\mA)\right|=hl_1(\mA)$. Since $(\mA)/nw_1(\mA)$ has $i-1$ northwest ribbons we can use induction and the $1-1$-correspondence given in Remark~\ref{bem:LRmain} and the claim holds true.

Let us now assume that the claim holds true if $l(\pinw(\mA))=i$ and $\mA$ decays into $j-1$ disconnected skew diagrams. Suppose $\mA$ decays into $j$ disconnected skew diagrams $\mB_i$. We can use the same argument as in the weak ribbon case (in the above argument we never used the fact that the $\mB_i$ are ribbons) to get $hl(\mA)=\pinw(\mA)$.
\end{proof}

\begin{Bem}\label{bem:gammaskew}
The above proof tells us also the exact shape of the $\nu$ with $hl(\nu)=hl(\mA)$ and $[\nu]\in[\mA]$. Again, as in Remark~\ref{bem:gamma} we construct a partition $\gamma$ such that the box $(i,i)$ has the same arm resp.\  leg length as $nw_i(\mA)$ and the box $(i,i)$ is in $\gamma$ if $\pinw(\mA)_i\not=0$. Let $nw_j(\mA)$ decay into $k_j$ disconnected ribbons. To obtain the partitions $\nu$ we have to add for each $j$ $k_j-1$ boxes to the $j$th row or column in $\gamma$. The number of ways to obtain $\nu$ from $\gamma$ by adding these boxes is then the multiplicity with which $[\nu]$ appears in $[\mA]$.
\end{Bem}
\begin{proof}
In the proof of Theorem~\ref{Sa:main} we have seen how to get the characters $[\nu]\in[\mA]$ with $hl(\n)=hl(\mA)$. To construct these $\n$ we have to maximize the first principal hook length and then the second and so on. Suppose we want to maximize the $j$-th principal hook length while the $1$-st to $j-1$-th principal hook lengths of $\n$ are maximal. If the $j$-th northwest ribbon decays into $k_j$ proper ribbons then the skew diagram obtained by removing the first $j-1$ northwest ribbons of $\mA$ also decays into $k_j$ disconnected skew diagrams $\mB_i$. We have then to calculate all the products of characters in different $[\mB_i]$ having maximal hook length partitions. But from Remark~\ref{bem:gamma} we know how to maximize the first hook length of partitions whose corresponding character is in this product (and so maximize the $j$-th hook length of partitions whose corresponding character is in $[\mA]$). We have to construct the hook having as arm resp.\ leg length the sum of the arm resp.\ leg lengths of the first principal hooks of the partitions whose corresponding characters are multiplied and then add a box to either the first row or column and then iterate this.
\end{proof}

Since there are $\binom{a}{b}$ ways to choose from $a$ boxes $b$ boxes and put them into a row and put the other $a-b$ boxes into a column we get the following Proposition:

\begin{pro}\label{Le-Prop:c-value}
 Let $nw_i(\l/\mu)$ decay into $k_i$ disconnected ribbons.

Then there are $\prod_i k_i$ different characters $[\nu]\in[\l/\mu]$ with $hl(\nu)=hl(\l/\mu)$ and we have
\[c(\l;\mu,\nu)=\prod_i \binom{k_i -1}{\a_i} \]
where $\a_i$ is the number of boxes placed in the $i$-th row in the construction of $\nu$ from $\gamma$ (as in Remark~\ref{bem:gammaskew}).
\end{pro}

We give an example in which two of the $[\nu]\in[\mA]$ have multiplicity $2$. We take $\mA=(8^2,7,4,3^2)/(4,3,2)$. We have the following northwest ribbon decomposition, where the boxes are labeled $i$ if they are in $nw_i(\mA)$. We notice that $k_1=1,k_2=3,k_3=2$ and so expect $6$ different characters $\n_i$ with $hl(\n_i)=hl(\mA)$ and the value $2$ as highest multiplicity of one of these characters.
{\footnotesize\[\young(::::1111,:::11222,::11223,1112,122,123)\]}

If we follow the proof from Theorem~\ref{Sa:main}, we remove the first northwest ribbon and obtain $\mB=(7,6,3,2,2)/(4,3,2)$, which decomposes into $\mC_1=(2^2),\mC_2=(1)$ and $\mC_3=(4,3)/(1)$. To calculate $\mu$ with $[\mu]\in[\mB]$ and $hl(\mu)=hl(\mB)$ we have to multiply the characters with the maximal hook length partitions in the $[\mC_i]$, but for $i=1,2$ the $[\mC_i]$ are already those characters. To obtain the characters with maximal hook length partition in $[\mC_3]$ we can use Theorem~\ref{Sa:main} again. Removing the first northwest ribbon of $\mC_3$ gives the partition $(1)$ so the character corresponding to $\a=(4,2)$ is the only one with maximal hook length partition in $[\mC_3]$.

So we have to calculate the product knowing that we are only interested in those $[\mu]$ with $hl(\mu)=hl(\mB)$. So we first multiply $[1]$ with $[2^2]$ which gives us $[1]\otimes[2,2]=[3,2]+[2,2,1]$ where we had the choice to maximize the first row or column. To obtain the $[\mu]$ we now multiply both with $[4,2]$ and obtain:
\[[3,2]\otimes[4,2]=[7,3,1]+[7,2,2]+[6,3,1,1]+[6,2,2,1]+14\textnormal{ other characters} \]
\[[2,2,1]\otimes[4,2]=[6,3,1,1]+[6,2,2,1]+[5,3,1,1,1]+[5,2,2,1,1]+10\textnormal{ other characters} \] where we had the choice to maximize the first row or column and the second row or column. We also see that $[6,3,1,1]$ and $[6,2,2,1]$ appear with multiplicity $2$ in $[\mB]$, because we maximized once the first row and once the first column.

This tells us that the characters $[\nu]\in[\mA]$ with $hl(\nu)=hl(\mA)$ are those corresponding to the following partitions:
{\footnotesize\[\yng(8,8,4,2,1,1)\qquad\yng(8,8,3,3,1,1)\qquad\yng(8,7,4,2,2,1)\]}
{\footnotesize\[\yng(8,7,3,3,2,1)\qquad\yng(8,6,4,2,2,2)\qquad\yng(8,6,3,3,2,2).\]}
Here $[8,7,4,2,2,1]$ and $[8,7,3,3,2,1]$ appear with multiplicity $2$. Furthermore, since $[\mB]$ has $25$ different irreducible characters, we know that in $[\mA]$ there are $25$ different irreducible characters $[\xi]$ with $hl_1(\xi)=hl_1(\mA)$.

If we would follow the construction in Remark~\ref{bem:gammaskew} we construct first $\gamma$ with
{\footnotesize\[\gamma=\yng(8,6,3,2,1,1)\]}
and then add $k_1-1=0$ boxes to the first row or column, $k_2-1=2$ boxes to the second row or column and $k_3-1=1$ boxes to the third row or column with the same result.

We now want to show that the minimal Durfee size of all $[\mu]\in[\mA]$ is $l(hl(\mA))$. For this we need the following:

\begin{Le}
 Let $\mA$ be a skew diagram and set $H(i,j)=a$ if the box $(i,j)$ belongs to $nw_a(\mA)$.

If $H(i+1,j+1)>1$ then $H(i+1,j+1)=H(i,j)+1$.
\end{Le}
\begin{proof}
Let $H(i+1,j+1)=b$ and let $\mB$ be the skew diagram where the first $b-2$ northwest ribbons are removed from $\mA$. If the box $(i,j)$ is not in $\mB$ then the box $(i+1,j+1)$ would belong to the $b-1$th northwest ribbon of $\mA$ which is not the case. So $(i,j)$ is in $\mB$ and so must belong to the $b-1$th northwest ribbon of $\mA$ and so we have $H(i,j)=b-1=H(i+1,j+1)-1$.
\end{proof}

\begin{pro}\label{Le-prop:minD}
Let $[\mA]$ be a skew character.

The $[\nu]\in[\mA]$ with $hl(\nu)=hl(\mA)$ have minimal Durfee size of all $[\mu]\in[\mA]$. In particular the minimal Durfee size of a character $[\mu]\in[\mA]$ is $l(hl(\mA))$.
\end{pro}
\begin{proof}
 Set $h=l(hl(\mA))$. The previous lemma tells us that if a box  belongs to $nw_h(\mA)$ then it lies in the southeastern corner of a square $(h^h)$ which lies completely in $\mA$. But if the square $(h^h)$ lies in $\mA$ then the square $(h^h)$ lies also in all partitions $\m$ whose corresponding character appears in $[\mA]$  and so $h$ is a lower bound for the Durfee size of a character $[\mu]\in[\mA]$.
But since the $[\nu]\in[\mA]$ with $hl(\nu)=hl(\mA)$ have Durfee size $d(\nu)=h$ the claim holds true.
\end{proof}

\section{Maximal Durfee Sizes In Skew Characters}\label{Se:maxD}
In this section we use Theorem~\ref{Sa:main} and \cite[Theorem 4.2]{Gut} to determine for a product of two characters and for some special skew characters the maximal Durfee size of characters and explicitly construct some characters with maximal Durfee size.

\cite[Theorem 4.2]{Gut} states the following:

\begin{Sa}
\label{Sa:Mainskewschub}
Let $\mu,\l$ be partitions with $\mu\subseteq\l\subseteq (k^l)$ with some fixed integers $k,l$. Set $\l^{-1}=(k^l)/\l$.

Then: The coefficient of $[\a]$ in $[\l/\mu]$ equals the coefficient of $[\a^{-1}]=[(k^l)/\a]$ in $[\mu]\star_{(k^l)}[\l^{-1}].$
\end{Sa}

We will use that for $k\geq\mu_1+\nu_1,l\geq l(\mu)+l(\nu)$ the Schubert-Product $[\mu]\star_{(k^l)}[\nu]$ is the ordinary product $[\mu]\otimes[\nu]$.

Let us associate a skew diagram $\mA=\left((m^m)/\a\right)^\circ)/\b$ to partitions $\a,\b$ with $m=\max(\a_1+\b_1,l(\a)+l(\b))$. To obtain $\mA$ we remove from the square $(m^m)$ the partition $\b$ as usual and the partition $\a$ rotated by $180^\circ$ from the lower right corner.

Theorem~\ref{Sa:Mainskewschub} tells us that characters $[\nu]$ in $[\mA]$ correspond to characters $[(m^m)/\nu]$ in the product $[\a]\otimes[\b]$. This means that characters with maximal Durfee size in the product $[\a]\otimes[\b]$ correspond to the characters with minimal Durfee size in its associated skew character $[\mA]$. But from the previous section we know some characters with minimal Durfee size.

So from Theorem~\ref{Sa:main} we get:

\begin{pro}
 Let $\a,\b$ be partitions, $m=\max(\a_1+\b_1,l(\a)+l(\b)),\mA=\left((m^m)/\a\right)^\circ/\b$.

Then for the product $[\a]\otimes[\b]$:
\begin{enumerate}
 \item $d([\a]\otimes[\b])=m-l(hl(\mA))$.
 \item Let $nw_i(\mA)$ decompose into $k_i$ disconnected ribbons. Then there are at least $\prod_i k_i$ different characters with maximal Durfee size in $[\a]\otimes[\b]$ and the highest multiplicity of a character with maximal Durfee size is at least $\prod_i\binom{k_i-1}{\lfloor \frac{k_i-1}{2} \rfloor}$.
 \item If $[\nu]\in[\mA]$ with $hl(\nu)=hl(\mA)$ then $[\nu^{-1}]=[(m^m)/\nu]\in[\a]\otimes[\b]$ has maximal Durfee size in $[\a]\otimes[\b]$.
\end{enumerate}
\end{pro}

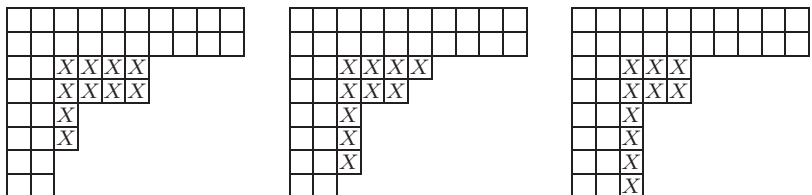
\begin{Bsp}
If we want to know for $\a=(5^2,3^2,2),\b=(4,3,1^2)$ some characters with maximal Durfee size in the product $[\a]\otimes[\b]$ we first construct the associated skew diagram $\mA$
{\footnotesize\[\mA=\young(:::::1111,:::::1222,:::111233,:::122234,::1123334,11122344,12223345,123334,12344)\]}
where the boxes have the entry $i$ if they belong to $nw_i(\mA)$. By Remark~\ref{bem:gammaskew} we have the following partitions with maximal principal hook length partition in $[\mA]$:
\newcommand{\X}{X}
{\footnotesize\[\nu_1=\young(\h\h\h\h\h\h\h\h\h,\h\h\h\h\h\h\h\h\h,\h\h\h\h\h\h\h\h\h,\h\h\h\h\h\h\X\X\X,\h\h\h\h\h,\h\h\h\h,\h\h\h,\h\h\h,\h\h\h)\qquad \nu_2=\young(\h\h\h\h\h\h\h\h\h,\h\h\h\h\h\h\h\h\h,\h\h\h\h\h\h\h\h\h,\h\h\h\h\h\h\X\X,\h\h\h\h\h,\h\h\h\h,\h\h\h\X,\h\h\h,\h\h\h)\]
\[\nu_3=\young(\h\h\h\h\h\h\h\h\h,\h\h\h\h\h\h\h\h\h,\h\h\h\h\h\h\h\h\h,\h\h\h\h\h\h\X,\h\h\h\h\h,\h\h\h\h,\h\h\h\X,\h\h\h\X,\h\h\h)\qquad \nu_4=\young(\h\h\h\h\h\h\h\h\h,\h\h\h\h\h\h\h\h\h,\h\h\h\h\h\h\h\h\h,\h\h\h\h\h\h,\h\h\h\h\h,\h\h\h\h,\h\h\h\X,\h\h\h\X,\h\h\h\X).\]}
The empty boxes in the $\nu_i$ form the partition $\gamma$ from Remark~\ref{bem:gammaskew}. By Proposition~\ref{Le-Prop:c-value} the multiplicities of $\nu_1$ and $\nu_4$ are $1$ and the multiplicities of $\nu_2$ and $\nu_3$ are $3$. So we have in $[\a]\otimes[\b]$ the characters $[\nu_i^{-1}]=[(9^9)/\nu_i]$ with maximal Durfee size and corresponding multiplicities:
{\footnotesize\[\nu_1^{-1}=\yng(6,6,6,5,4)\qquad\nu_2^{-1}=\yng(6,6,5,5,4,1)\]
\[\nu_3^{-1}=\yng(6,5,5,5,4,2)\qquad\nu_4^{-1}=\yng(5,5,5,5,4,3).\]}
But there are many more characters with maximal Durfee size and the characters $[7,6,5,4,3,1^2],[7,5^2,4,3,2,1],[7,6,5,4,3,2]$ and $[7,6,4^2,3,2,1]$ all appear with the highest multiplicity which is $13$.
\end{Bsp}

For a skew diagram $\l/\mu$ with $\l=(\l_1^k,\l_{k+1},\l_{k+2},\ldots,\l_l)$,  Theorem~\ref{Sa:Mainskewschub} together with Theorem~\ref{Sa:discmain} and Remark~\ref{bem:minDdisc} gives us some characters $[\a]\in[\l/\mu]$ with maximal Durfee size if $\l_1=l, k\geq l(\mu)$ and $\mu_1\leq \l_l$.

So we get the following:
\begin{pro}
 Let $\l,\mu$ be partitions with $\l=(\l_1^k,\l_{k+1},\l_{k+2},\ldots,\l_l), \l_1=l, k\geq l(\mu)$ and $\mu_1\leq \l_l$. We set $\l^{-1}=(l^l)/\l$. Then:

\begin{enumerate}
\item $d(\l/\mu)=l-\max(d(\mu),d(\l^{-1}))$.

\item There are at least $2^{\min(d(\mu),d(\l^{-1}))}$ different characters with maximal Durfee size in $[\l/\mu]$ and at least $2^{\min(d(\mu),d(\l^{-1}))}$ of them appear with multiplicity $1$.

\item If $[\a]\in[\mu]\otimes[\l^{-1}]$ with $hl(\a)=hl([\mu]\otimes[\l^{-1}])$ then $[\a^{-1}]=[(m^m)/\a]\in[\l/\mu]$ has maximal Durfee size in $[\l/\mu]$.
\end{enumerate}

\end{pro}

\section{On The Equality Of Skew Characters}\label{Se:eq}

The problem under which conditions two skew diagrams give rise to the same skew character has recently seen much work (see for example \cite{MW} or \cite{RSW}).

We can use the theorems and remarks in Section~\ref{Se:main} to give us conditions for two skew diagrams $\mA,\mB$ to represent the same skew characters. In summary we have the following:

\begin{Sa}\label{sa:cond}
Let $\mA,\mB$ be skew diagrams.

If $[\mA]=[\mB]$ then the following holds true:
\begin{enumerate}
 \item $\pinw(\mA)=\pinw(\mB)$
 \item For every $i$ the numbers of ribbons into which $nw_i(\mA)$ and $nw_i(\mB)$ decompose are the same.
 \item For every $i$ the arm resp.\  leg length of $nw_i(\mA)$ and $nw_i(\mB)$ are the same.
 \item For every $i$, if we remove the first $i$ northwest ribbons from $\mA$ resp.\  $\mB$ to get $\tilde\mA$ resp.\  $\tilde\mB$ then $[\tilde\mA]=[\tilde\mB]$.
\end{enumerate}
\end{Sa}

We want to use Theorem~\ref{sa:cond} to check if the skew diagrams  $\mA=(10^2,8^4,5^2)/(5^4)$ and $\mB=(10^4,8^2,3^2)/(5^4)$ given in Example~\ref{ex1} give rise to the same skew character. We label the boxes contained in $nw_i$ with $i$ and have the following situation:

{\footnotesize\[ \mA=\young(:::::11111,:::::12222,:::::123,:::::123,11111123,12222223,12333,12344) \qquad \mB=\young(:::::11111,:::::12222,:::::12333,:::::12344,11111123,12222223,123,123). \]}

We see that the parts $1-3$ of Theorem~\ref{sa:cond} hold true and investigate $4$. If we remove the first three northwest ribbons the remaining skew diagram is in both cases the partition $(2)$. If we remove only the first two northwest ribbons we get $\tilde\mA=(4^4,3^2)/(3^4)$ and $\tilde\mB=(4^2,2^2,1^2)/(1^4)$. We have
\begin{align*}
 [\tilde\mA]&=[4^2,1^2]+[4,3,1^3]+[3^2,1^4]\\
 [\tilde\mB]&=[4^2,1^2]+[4,3,1^3]+[3^2,1^4]+[4,3,2,1]+[3^2,2^2]+[3^2,2,1^2]
\end{align*}
and so $[\mA]\not=[\mB]$.

The decomposition of $[\tilde\mA]$ and $[\tilde\mB]$ gives us the following partitions whose characters appear with multiplicity $1$ in both $[\mA]$ and $[\mB]$

{\footnotesize
\[\young(\h\h\h\h\h\h\h\h\h\h,\h\h\h\h\h\h\h\h\h\h,\h\h XXXX,\h\h XXXX,\h\h X,\h\h X,\h\h,\h\h)\qquad
\young(\h\h\h\h\h\h\h\h\h\h,\h\h\h\h\h\h\h\h\h\h,\h\h XXXX,\h\h XXX,\h\h X,\h\h X,\h\h X,\h\h)\qquad
\young(\h\h\h\h\h\h\h\h\h\h,\h\h\h\h\h\h\h\h\h\h,\h\h XXX,\h\h XXX,\h\h X,\h\h X,\h\h X,\h\h X)
\]}
and the partitions whose corresponding characters appear with multiplicity $1$ in $[\mB]$ but not in $[\mA]$
{\footnotesize
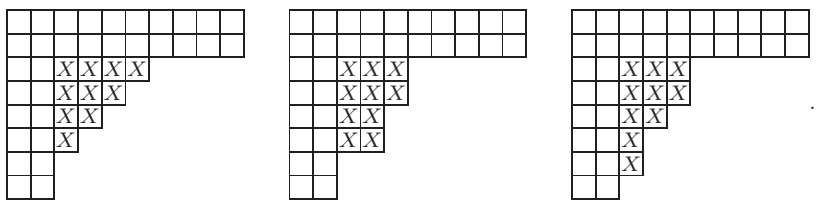
\[\young(\h\h\h\h\h\h\h\h\h\h,\h\h\h\h\h\h\h\h\h\h,\h\h XXXX,\h\h XXX,\h\h XX,\h\h X,\h\h,\h\h)\qquad
\young(\h\h\h\h\h\h\h\h\h\h,\h\h\h\h\h\h\h\h\h\h,\h\h XXX,\h\h XXX,\h\h XX,\h\h XX,\h\h ,\h\h)\qquad
\young(\h\h\h\h\h\h\h\h\h\h,\h\h\h\h\h\h\h\h\h\h,\h\h XXX,\h\h XXX,\h\h XX,\h\h X,\h\h X,\h\h )
.\]}

These are all characters appearing in $[\mA]$ or $[\mB]$ with maximal first and second principal hook length.

{\bfseries Acknowledgement.} I am very grateful to Christine Bessenrodt for her support and fruitful discussions.

\end{document}